\documentclass{svmult}
\pdfoutput=1
\usepackage{graphicx,url}
\usepackage[all]{xy}

\usepackage{amsfonts}
\usepackage{amssymb,amsmath,amscd}

\newdimen\hoogte    \hoogte=12pt    
\newdimen\breedte  \breedte=14pt  
\newdimen\dikte    \dikte=0.5pt    
\def\beginYoung{
        \begingroup
        \def\vr{\vrule height0.8\hoogte width\dikte depth 0.3\hoogte}
        \def\fbox##1{\vbox{\offinterlineskip
                    \hrule height\dikte
                    \hbox to \breedte{\vr\hfill##1\hfill\vr}
                    \hrule height\dikte}}
        \vbox\bgroup \offinterlineskip \tabskip=-\dikte \lineskip=-\dikte
            \halign\bgroup &\fbox{##\unskip}\unskip  \crcr }
\def\End@Young{\egroup\egroup\endgroup}

\newdimen\breedtelong  \breedtelong=28pt  

\def\beginYounglong{
        \begingroup
        \def\vr{\vrule height0.8\hoogte width\dikte depth 0.3\hoogte}
        \def\fboxlong##1{\vbox{\offinterlineskip
                    \hrule height\dikte
                    \hbox to \breedtelong{\vr\hfill##1\hfill\vr}
                    \hrule height\dikte}}
        \vbox\bgroup \offinterlineskip \tabskip=-\dikte \lineskip=-\dikte
            \halign\bgroup &\fboxlong{##\unskip}\unskip  \crcr }
\def\Endlong@Younglong{\egroup\egroup\endgroup}

\usepackage{amsfonts,amssymb,amsmath,,amscd}
\newcommand{\C}{\ensuremath{\mathbb{C}}}

\newcommand{\R}{\ensuremath{\mathbb{R}}}

\date{\today}
\begin{document}

\title*{Between two moments}

\author{Alain Chenciner\inst{1}\and 
Bernard Leclerc\inst{2}}
\institute{Observatoire de Paris, IMCCE (UMR 8028), ASD,
77, avenue Denfert-Rochereau, F75014 Paris  \& D\'epartement de math\'ematique, Universit\'e Paris 7,
\texttt{chenciner@imcce.fr}
\and Universit\'e de Caen, LMNO (UMR 6139), Campus 2, F14032 Caen Cedex,
\texttt{bernard.leclerc@unicaen.fr}}

\maketitle

\hangindent=6cm\hangafter =-9 \noindent{\it Dedicated with admiration to the\\ memory of Jean-Marie Souriau.}\medskip

\abstract In this short note, we draw
attention to a relation between two Horn polytopes which is proved in \cite{CJ} as the result on the one side of a deep combinatorial result in \cite{FFLP}, on the other side of a simple computation involving complex structures. This suggested an inequality between Littlewood-Richardson coefficients which we prove using the symmetric characterization of these coefficients given in \cite{CL}.
\medskip

\section{Moment maps and Horn's problem}

Recall (see for example \cite{K}) that the mapping $H\mapsto [k\mapsto Trace(iHk)]$ is a natural identification of the space $\mathcal H$ of Hermitian matrices with the dual $u(n)^*$ of the Lie algebra $u(n)$ of the unitary group $U(n)$. With this identification, the coadjoint action of $U(n)$ on $u(n)^*$ becomes the action by conjugation on the space of Hermitian matrices and the orbits ${\mathcal O}_\lambda$ are uniquely characterized by the common spectrum $\lambda$ of their elements. A coadjoint orbit being endowed with the Kostant-Souriau symplectic structure, the moment map becomes the canonical inclusion ${\mathcal O}_\lambda\subset\mathcal H$. 
\smallskip 

\noindent It follows from standard properties of moment maps (see again \cite{K}) that with this identification, the moment map of the diagonal action of $U(n)$ on 
${\mathcal O}_\lambda\times {\mathcal O}_\mu$ is simply the mapping $(A,B)\mapsto A+B$. This gives the relation with the so-called {\it Horn's problem} of describing the inequalities which constrain the spectrum of the sum of two hermitian matrices whose spectra are given (see the beautiful review \cite{F}).

\noindent Kirwan's theorem asserts that, in the case of the diagonal action above, the intersection of the image of the moment map with the principal Weyl chamber $W_n^+=\nu=\{\nu_1\ge\nu_2\ge\cdots\ge\nu_n\}$, that is the set  of ordered spectra of matrices of the form $A+B$ where 
$A\in{\mathcal O}_\lambda$ and $B\in{\mathcal O}_\mu$, is a convex polytope, which we shall call the {\it Horn polytope}. Moreover, Horn's conjecture, whose proof follows from the works of Klyachko on the one hand, Knutson and Tao on the other hand (see \cite{F}), gives inequalities defining the faces of the Horn polytope.
\smallskip

\noindent Finally, replacing
Hermitian matrices by
symmetric matrices and the unitary group by the orthogonal group does not change the Horn polytope, namely: the set of ordered spectra $\nu$ of the sums $A+B$ of matrices with given spectra $\lambda$ and $\mu$ is the same whether $A$ and $B$ are hermitian or real symmetric (see \cite{F}).

\section{Complex structures and a Horn-type problem} Motivated by a study of the angular momenta of rigid body motions in dimensions higher than 3 (see \cite{C, CJ}),
let us consider the following ``Horn-like" problem: characterize the set $\mathcal Q$ (noted $\hbox{Im}{\mathcal F}$ in \cite{CJ}) of ordered spectra of sums 
$$J^{-1}S_0J+S_0,$$ 
where $S_0$ is a given 
$2p\times 2p$ real symmetric matrix and $J\in  O(2p)$ is a complex structure, that is an isometry of $\R^{2p}$ (with its standard euclidean structure) such that $J^2=-\mathrm{Id}$. 
\smallskip

\noindent Replacing the set of complex structures $J$ by the full group of isometries $O(2p)$, one gets exactly the Horn polytope in the case when $n=2p$ and the two spectra $\lambda$ and $\mu$ coincide with the one of $S_0$. 
\smallskip

\noindent The main result of \cite{CJ} is that $\mathcal Q$ is a convex polytope. The proof is not direct: one encloses $\mathcal Q$ between two convex polytopes $\mathcal P_1$ and $\mathcal P_2$ associated with two Horn's problems, one in dimension $p$ and one in dimension $2p$, and one deduces from results due to Fomin, Fulton, Li and Poon \cite{FFLP} and based on deep combinatorial lemmas due to Carr\'e and Leclerc \cite{CL}, that $\mathcal P_1$ and $\mathcal P_2$ coincide, and hence coincide with $\mathcal Q$.

\section{The two polytopes ${\mathcal P}_1$ and ${\mathcal P}_2$}

\noindent Given decreasing sequences 
$$\sigma=(\sigma_1\ge\sigma_2\ge\cdots\ge\sigma_{2p})\quad\hbox{and}\quad \nu=(\nu_1\ge\cdots\ge\nu_p)$$ 
of real numbers, one defines
\begin{equation}
\left\{
\begin{split}
&\sigma_-=(\sigma_1,\sigma_3,\cdots,\sigma_{2p-1}),\quad \sigma_+=(\sigma_2,\sigma_4,\cdots,\sigma_{2p}),\\
&\nu^{(2)}=(\nu_1,\nu_1,\nu_2,\nu_2,\cdots,\nu_p,\nu_p).
\end{split}
\right.
\end{equation}
\smallskip

\noindent 
Now, ${\mathcal P}_1$ is the Horn polytope for $p\times p$ symmetric matrices $a+b$ with $\hbox{spectrum\,}(a)=\sigma_-$ and $\hbox{spectrum\,}(b)=\sigma_+$, while ${\mathcal P}_2$ is the intersection with the set $\Delta$ of ``hermitian spectra" (that is spectra of the form $\nu^{(2)}$) of the Horn polytope for  $2p\times 2p$ symmetric matrices $A+B$ with $\hbox{spectrum\,} (A)=\hbox{spectrum\,} (B)=\sigma$. \smallskip

\section{The inclusion ${\mathcal P}_1\subset{\mathcal P}_2$}

The inclusion ${\mathcal P}_1\subset{\mathcal P}_2$ is elementary: it  follows from the inclusions
\begin{equation}
{\mathcal P}_1\subset {\mathcal Q}\subset{\mathcal P}_2,
\end{equation}
where ${\mathcal Q}$ is the set of ordered spectra of matrices of the form $S+J^{-1}SJ$, where $S=\hbox{diag\,}\sigma$ and $J$ is a complex structure:
$$
J=R^{-1}J_0R,\; R\in SO(2p), J_0=\begin{pmatrix}0&-\mathrm{Id}\\\mathrm{Id}&0\end{pmatrix}.
$$

\noindent The second inclusion is obvious (replace $J$ by any element $R\in SO(2p)$ and look for the spectra of sums $S+R^{-1}SR$ which are $J$-hermitian for some $J$;
\smallskip

\noindent The first one comes from the fact that ${\mathcal P}_1$ is exactly the subset of $\hbox{Im\,}{\mathcal F}$ obtained when one takes into account only the complex structures $J$ which send the subspace of $\R^{2p}$ generated by the basis vectors $\vec e_1,\vec e_3,\cdots,\vec e_{2p-1}$ onto the orthogonal subspace, generated by$\{\vec e_2,\vec e_4,\cdots,\vec e_{2p}\}$.
More precisely, it comes from the following identity where $\rho\in SO(p)$:
\begin{equation}
\begin{cases}
&\begin{pmatrix}
\sigma_-&0\\
0&\sigma_+
\end{pmatrix}+
\begin{pmatrix}
0&-\rho^{-1}\\
\rho&0
\end{pmatrix}^{-1}
\begin{pmatrix}
\sigma_-&0\\
0&\sigma_+
\end{pmatrix}
\begin{pmatrix}
0&-\rho^{-1}\\
\rho&0
\end{pmatrix}\\
=&
\begin{pmatrix}
\sigma_-+\rho^{-1}\sigma_+\rho&0\\
0&\rho\sigma_-\rho^{-1}+\sigma_+
\end{pmatrix}\cdot
\end{cases}
\end{equation}

\section{Littlewood-Richardson coefficients}

{\it From now on, the decreasing sequences that we consider consist in non negative integers, which allows them to take the name of ``partitions".
One calls $|\lambda|=\sum{\lambda_i}$ the ``length" of the partition.}
\smallskip

\noindent The so-called Littlewood-Richardson coefficients have their origin in the decomposition into irreducible holomorphic representations of the tensor product of two irreducible representations of the linear group $GL(n,\C)$. For their definition and  their many avatars, as well as for the computational rule expressing them in terms of the Ferrer-Young diagrams representing partitions, we refer to \cite{F}. 
\smallskip

Let 
\begin{equation*}
\begin{split}
&\alpha=\{\alpha_1\ge\alpha_2\ge\cdots\ge\alpha_n\},\\
&\beta=\{\beta_1\ge\beta_2\ge\cdots\ge\beta_n\},\\
&\gamma=\{\gamma_1\ge\gamma_2\ge\cdots\ge\gamma_n\},
\end{split}
\end{equation*}
be three partitions.
\smallskip

\noindent It results from the proof of Horn's conjecture that $\alpha, \beta,\gamma$ are respectively the spectra of $n\times n$ hermitian (or real symmetric) matrices $A,B,C$ with $C=A+B$ if and only if (\cite{F}, theorem 11)
\smallskip

1)$|\gamma|=|\alpha|+|\beta|$, 
\smallskip

2) the Littlewood-Richardson coefficient $c_{\alpha\beta}^\gamma$ does not vanish.
\smallskip

\noindent Hence, identifiying $\Delta$ with a subset of $\R^p$ by the mapping $\nu^{(2)}\mapsto \nu$:
\begin{equation}
\left\{
\begin{split}
{\mathcal P}_1&=\left\{\nu=(\nu_1\ge\nu_2\ge\cdots\ge\nu_p), c_{\sigma_-\sigma_+}^\nu\not=0\right\},\\
{\mathcal P}_2&=\left\{\nu=(\nu_1\ge\nu_2\ge\cdots\ge\nu_p), c_{\sigma\sigma}^{\nu^{(2)}}\not=0\right\}.
\end{split}
\right.
\end{equation}

\section{Yamanouchi tableaux and the Carr\'e-Leclerc formula}

{\it For the convenience of the reader, we reproduce here definitions given in sections 4.1 and 4.2 of \cite{FFLP}.}
\smallskip

\noindent It is classical to define a bijection $\Lambda$ between strictly increasing sequences $I=(i_1<i_2<\cdots <i_r)$ and partitions $\Lambda(I)=(\lambda_1\ge\cdots\ge\lambda_r)$ by the formul\ae
$$\lambda_1=i_r-r, \cdots,\lambda_r=i_1-1.$$
Given two strictly increasing sequences, one defines
$$\tau(I,J) =(2i_1-1,\cdots,2i_r-1)\cup(2j_1,\cdots, 2j_r)\; \hbox{reordered}.$$
To a partition $\lambda$ is associated a {\it Young diagram} of {\it shape} $\lambda$: this is a collection of boxes arranged in rows whose lengths from top to bottom (anglo-saxon convention) are $\lambda_1,\lambda_2,\cdots$ 
The induced operation on partitions,  $(\lambda,\mu)\mapsto \tau(\lambda,\mu)$  is described in \cite{FFLP} in terms of Young diagrams in the following way: ``If one traces the Young diagram of a partition by a sequence of horizontal and vertical steps moving from Southwest to Northeast in a rectangle containing the diagrams of $\lambda$ and $\mu$, the diagram of $\tau(\lambda,\mu)$ is traced in a rectangle twice as wide in both directions, by alternating steps from $\lambda$ and $\mu$, starting with the first step of $\lambda$". 
 It is also recalled in this paper that partitions $\tau(\lambda,\mu)$ are exactly the ones which correspond to {\it domino-decomposable} Young tableaux, that is tableaux which can be partitioned into disjoint $1\times 2$ or $2\times 1$ rectangles (the so-called {\it dominoes}).
A (semi-standard) {\it domino tableau} is such a decomposition of a tableau into dominos with a labelling rule of the dominos similar to the one of ordinary semi-standard tableaux: labels weakly increase from left to right along lines and strictly increase along columns from top to bottom.  
Special domino tableaux, called {\it Yamanouchi}, play the leading part: let the {\it reading word} of a tableau be obtained by listing the labels column by column from right to left and from top to bottom (a horizontal domino is skipped the first time it is encountered); a tableau is called {\it Yamanouchi} if every entry $i$ appears in any initial segment of its reading word at least as many times as any entry $j>i$.
 In \cite{CL}, the following rule for computing the Littlewood-Richardson coefficients is given in which, in contrast with the original Littlewood-Richardson rule, the lower indices play symmetric roles: let the {\it weight} $\nu$  of a Yamanouchi tableau be the sequence $(\nu_1,\cdots,\nu_n)$ where $\nu_i$ is the number of labels equal to $i$.  Note that $\nu_1\ge\cdots\ge\nu_n$.
  
\begin{proposition}[\cite{CL} Corollary 4.4]\label{LR=YDT} The coefficient $c^\nu_{\lambda\mu}$ is equal to the number of Yamanouchi domino tableaux of shape $\tau(\lambda,\mu)$ and weight $\nu$.
\end{proposition}
Pay attention that, while $c^\nu_{\lambda\mu}=c^\nu_{\mu\lambda}$,  the diagrams  $\tau(\lambda,\mu)$ and $\tau(\mu,\lambda)$ have in general different shapes. This will be important in the sequel.

\section{Strengthening the inclusion ${\mathcal P}_1\subset{\mathcal P}_2$}

The following proposition strengthens the inclusion ${\mathcal P}_1\subset{\mathcal P}_2$ while it was suggested by it:
\begin{proposition}\label{inequality} Given partitions $\sigma=\{\sigma_1,\sigma_2,\cdots,\sigma_{2p}\}$ and  
$\nu=\{\nu_1,\nu_2,\cdots,\nu_p\}$, we have
$$c_{\sigma_-\sigma_+}^\nu\le c_{\sigma\sigma}^{\nu^{(2)}}.$$
\end{proposition}
\begin{proof}
As in \cite{FFLP}, we use Proposition \ref{LR=YDT}.
We must define an injection of the set of Yamanouchi domino tableaux 
of shape $\tau(\sigma_-,\sigma_+)$ (or of shape $\tau(\sigma_+,\sigma_-)$) 
and weight $\nu$ into the set of 
Yamanouchi domino tableaux of shape $\tau(\sigma,\sigma)$ and weight $\nu^{(2)}$. 
In order to prove the above lemma,  the choice of $\tau(\sigma_+,\sigma_-)$ is much more convenient 
because  it is simply related to 
$\tau(\sigma,\sigma)$ while this is not the case of $\tau(\sigma_-,\sigma_+)$, more precisely, if
$\sigma=(\sigma_1,\sigma_2,\cdots,\sigma_{2p})$, 
\begin{equation*}
\begin{split}
\tau(\sigma_+,\sigma_-)&=(2\sigma_1,2\sigma_2,\cdots,2\sigma_{2p}),\\
\tau(\sigma,\sigma)&=(2\sigma_1,2\sigma_1,2\sigma_2,2\sigma_2,\cdots,2\sigma_{2p},2\sigma_{2p}).
\end{split}
\end{equation*}
\end{proof}
\noindent In words, the tableau $\tau(\sigma,\sigma)$ is obtained from the tableau $\tau(\sigma_+,\sigma_-)$ by duplicating each line. The injection we are looking for is then simply obtained by dividing vertically each domino of $\tau(\sigma_+,\sigma_-)$ and numbering the resulting two pieces respectively $2k-1$ and $2k$ if the original domino was numbered $k$. One checks immediately that the Yamanouchi property is verified. 
An example is given in \S\ref{sect-examples} below, see Figure~\ref{figure1} and Figure~\ref{figure2}.

\section{Concluding remarks}

1) As was said at the beginning, the inclusion ${\mathcal P}_1\supset{\mathcal P}_2$, and hence the  equality ${\mathcal P}_1={\mathcal P}_2$,  is proved in \cite{CJ}. It results from  the stronger inclusion  in ${\mathcal P}_1$ of the orthogonal projection on $\Delta$ of the Horn polytope
$$
{\mathcal P}=\left\{\gamma=(\gamma_1\ge\cdots\ge\gamma_{2p}),\; c_{\sigma\sigma}^{\gamma}\not=0\right\}.
$$ 
This last inclusion follows from the inequalities
$c^\nu_{\lambda\mu}\le c^{\tau(\nu,\nu)}_{\tau(\lambda,\mu)\tau(\lambda,\mu)}$, 
valid as soon as the lengths of the partitions at stake satisfy $|\nu|=|\lambda|+|\mu|$,
(\cite{FFLP} Proposition 4.5). The main ingredient of the proof is again the symmetric characterization of  the Littlewood-Richardson coefficients in terms of Yamanouchi tableaux given in \cite{CL}.
\smallskip

\noindent 
The equality ${\mathcal P}_1={\mathcal P}_2$ of the two polytopes has the consequence, not implied by the inequality of lemma \ref{inequality}, 
that  
$$c_{\sigma\sigma}^{\nu^{(2)}}\not=0\quad\hbox{ implies}\quad c_{\sigma_-\sigma_+}^\nu\not=0.$$ 
Nevertheless, the inequality given by Proposition~\ref{inequality} may be strict. An example is shown in \S\ref{sect-examples}, Figure~\ref{figure3}.
\bigskip

\noindent 
2) Dividing $\sigma$ into two partitions $\lambda$ and $\mu$ of length $p$ in an arbitrary way leads to polytopes strictly smaller than ${\mathcal P}_1={\mathcal P}_2$. Correspondingly, we have the following inequalities strengthening these inclusions:
$$c^\nu_{\lambda\mu}\le c^\nu_{\sigma_-\sigma_+}.$$
These inequalities are a special case of Corollary 14 of \cite{LPP}.

\section{Questions}

1) Find a more conceptual relation between the complex structures and  the doubling of Young tableaux. 

\noindent 
2) Find a direct proof of the fact that $\mathcal Q$ is convex, resp. a convex polytope.

\noindent 
3) Find a relation between the inequality which is the object of Proposition~\ref{inequality} and the one in \cite{FFLP} (Proposition 4.5). In the first case, each domino is divided into 2 pieces while in the second one it is divided into 4 pieces.

\section{Examples} \label{sect-examples} 
We take $\sigma=(5,3,2,0)$, so that $\sigma_+ = (3,0)$ and $\sigma_- = (5,2)$.
We then have 
\[
\tau(\sigma_+,\sigma_-)=(10,6,4,0), \qquad \tau(\sigma,\sigma)=(10,10,6,6,4,4,0,0).
\]
In Figure~\ref{figure1} below, we display the Yamanouchi domino tableaux $T_1$, $T_2$, $T_3$, $T_4$ of shape 
$\tau(\sigma_+,\sigma_-)$, and of respective weights
\[
\nu_{1} = (5,5),\quad 
\nu_{2} = (6,4),\quad 
\nu_{3} = (7,3), \quad
\nu_{4} = (8,2).
\]
Their respective reading words are
\[
w_1 = 1112212212,\quad
w_2 = 1112212112,\quad
w_3 = 1112112112,\quad
w_4 = 1111112112.
\]
In Figure~\ref{figure2} we display the corresponding Yamanouchi domino tableaux $U_1$, $U_2$, $U_3$, $U_4$ 
of shape $\tau(\sigma,\sigma)$,
obtained from $T_1$, $T_2$, $T_3$, $T_4$ via the duplication procedure of Proposition~\ref{inequality}. 

Finally, Figure~\ref{figure3} gives an example which shows that the inequality of Proposition~\ref{inequality} may be strict.
Here we take $\sigma = (7,6,4,3)$ and $\nu = (10,8,2)$, and we exhibit a Yamanouchi domino tableau $T$ of
shape $\tau(\sigma,\sigma)$ and weight $\nu^{(2)}$ which cannot be obtained from a Yamanouchi domino tableau
of shape $\tau(\sigma_+,\sigma_-)$ and weight $\nu$ by means of the duplication procedure described in the proof
of Proposition~\ref{inequality}.

\pagebreak

\begin{figure}
\begin{center}
\includegraphics[scale=0.55]{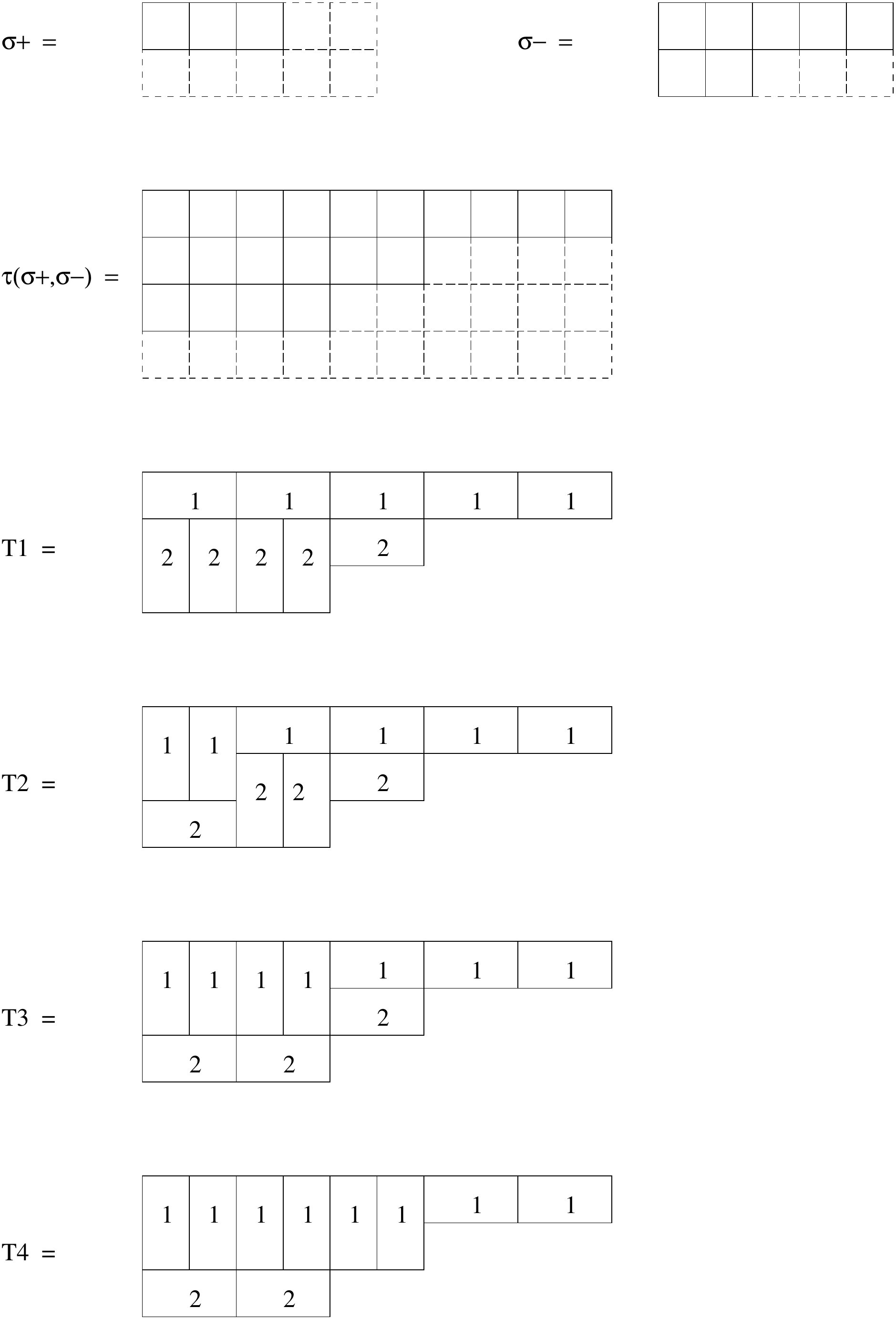}
\end{center}
\caption{\label{figure1} {Yamanouchi domino tableaux of shape $\tau(\sigma_+,\sigma_-)$.}}
\end{figure}

\begin{figure}
\begin{center}
\includegraphics[scale=0.55]{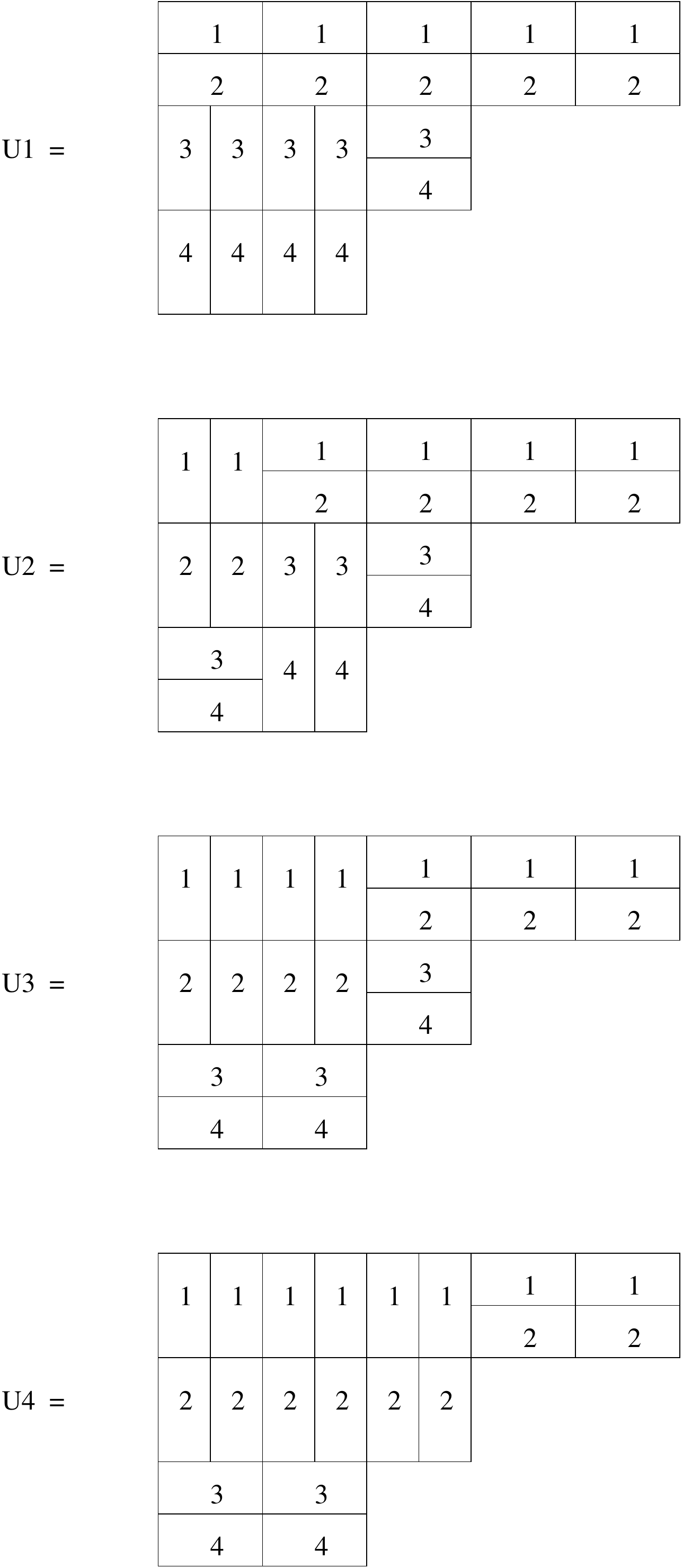}
\end{center}
\caption{\label{figure2} {Yamanouchi domino tableaux of shape $\tau(\sigma,\sigma)$.}}
\end{figure}

\begin{figure}
\begin{center}
\includegraphics[scale=0.55]{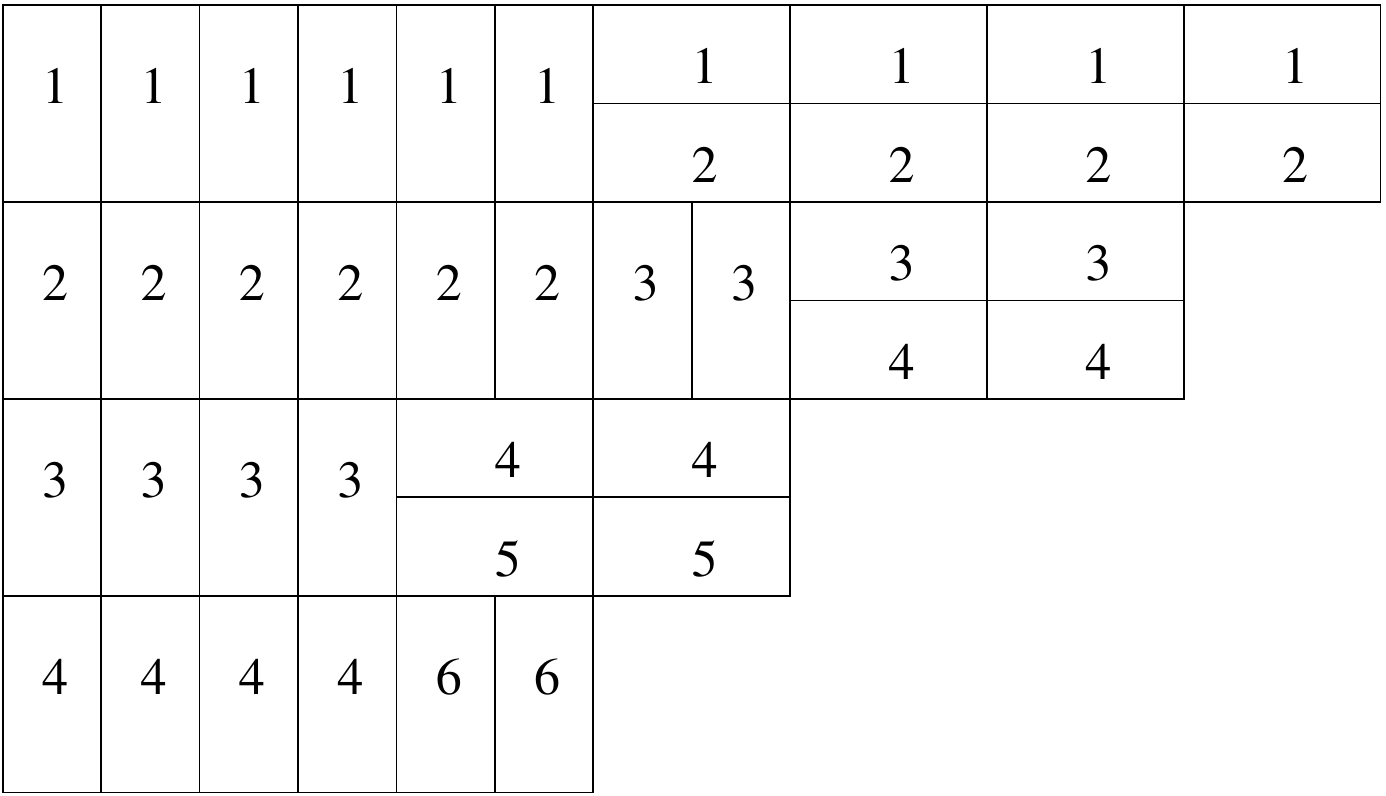}
\end{center}
\caption{\label{figure3} {A Yamanouchi tableau of shape $\tau(\sigma,\sigma)$ and weight $\nu^{(2)}$
which is not obtained by the duplication procedure of Proposition~\ref{inequality}.}}
\end{figure}

\end{document}